\documentclass[12pt]{amsart}

\usepackage{amsmath,amssymb,amsthm}
\usepackage{amscd}
\usepackage{color}
\newcommand{\QQ}{\mathbb{Q}}

\newcommand{\NN}{\mathbb{N}}

\newcommand{\CC}{\mathbb{C}}

\newcommand{\PP}{\mathbb{P}}

\newcommand{\ZZ}{\mathbb{Z}}

\newcommand{\la}{\langle}
\newcommand{\ra}{\rangle}

\newcommand{\CP}{\CC P}

\newcommand{\orb}{\mathrm{orb}}

\theoremstyle{plain}
\newtheorem{proposition}{Proposition}

\newtheorem{theorem}[proposition]{Theorem}

\theoremstyle{definition}
\newtheorem{definition}[proposition]{Definition}

\theoremstyle{remark}
\newtheorem{remark}[proposition]{Remark}
\newtheorem{question}{Question}

\title{Negative Sasakian structures on simply-connected $5$-manifolds}

\author[V. Mu\~{n}oz]{Vicente Mu\~{n}oz}
\address{Departamento de \'Algebra, Geometr\'ia y Topolog\'ia, Facultad de Ciencias, Universidad de M\'alaga, Campus de Teatinos s/n, 29071 Málaga, Spain}\email{vicente.munoz@uma.es}

\author[M. Sch\"utt]{Matthias Sch\"utt}
\address{Institut f\"ur Algebraische Geometrie, Gottfried Wilhelm Leibniz Universit\"at Hannover, Welfengarten 1, 30167 Hannover, Germany}
\email{schuett@math.uni-hannover.de}

\author[A. Tralle]{Aleksy Tralle}
\address{Faculty of Mathematics and Computer Science, University of Warmia and Mazury, S\l\/oneczna 54, 10-710 Olsztyn, Poland}
\email{tralle@matman.uwm.edu.pl}

\subjclass[2010]{53C25, 53D35, 14J27, 14J17}

\keywords{Sasakian, Smale-Barden manifold, Seifert bundle, singular surface}

\begin{document}

\begin{abstract}
We study several questions on the existence of negative Sasakian structures on simply connected rational homology spheres and on Smale-Barden manifolds of the form $\#_k(S^2\times S^3)$. First, we prove that any simply connected rational homology sphere admitting positive Sasakian structures also admits a negative one. This result answers the question, posed by Boyer and Galicki in their book \cite{BG}, of determining 
which simply connected rational homology spheres admit both negative and positive Sasakian structures. Second, we prove that the connected sum $\#_k(S^2\times S^3)$ admits negative quasi-regular Sasakian structures for any $k$. This yields a complete answer to another question posed in \cite{BG}. 
\end{abstract}  

\maketitle

%%%%%%%%%%%%%%%%%%%%%%%%%%%%%%%%%%%
\section{Introduction}
%%%%%%%%%%%%%%%%%%%%%%%%%%%%%%%%%%%
This work deals with Sasakian manifolds. The basic definitions and facts regarding this structure are given in Section \ref{sec:basic-def}. 
Recall that the Reeb vector field $\xi$ on a co-oriented contact manifold $(M,\eta)$ determines a $1$-dimensional foliation 
$\mathcal{F}_{\xi}$ called the {\it characteristic foliation}. If we are given a  manifold $M$ with a Sasakian structure $(\eta,\xi,\phi, g)$, then 
one can define  {\it basic Chern classes} $c_k(\mathcal{F}_{\xi})$ of $\mathcal{F}_{\xi}$ which are elements of the 
basic cohomology $H^{2k}_B(\mathcal{F}_{\xi})$ (see \cite[Theorem/Definition 7.5.17]{BG}).     
We say that a Sasakian structure is positive (negative) if $c_1(\mathcal{F}_{\xi})$ can be represented by a positive 
(negative) definite $(1,1)$-form. A Sasakian structure is called null, if $c_1(\mathcal{F}_{\xi})=0$. If none of these, it is called indefinite. 

The following problems are formulated in the seminal book on Sasakian geometry of Boyer and Galicki \cite{BG}.

\begin{question}[{\cite[Open Problems 10.3.3 and 10.3.4]{BG}}]\label{quest:bg-main}  
Which simply connected rational homology $5$-spheres admit negative Sasakian structures?
\end{question}

To put this question into a broader context, let us recall the foundational results of Koll\'ar \cite{K}, which yield
the full structural description of rational homology spheres $M$ with $H_1(M,\ZZ)=0$ and  admitting {\it positive} Sasakian structures.

\begin{theorem}[{\cite[Theorem 9.1]{K}}]  \label{thm:Ko}
Let $M\rightarrow (X,\Delta)$ be a $5$-dimensional Seifert bundle, with $M$ smooth, where $\Delta =\sum \big(1-\frac{1}{m_i}\big) D_i$ is the
branch divisor. Then
\begin{enumerate}
\item If $M$ is a rational homology sphere with $H_1(M,\ZZ)=0$, then
 \begin{enumerate}
 \item $X$ has only cyclic quotient singularities, and $H_2(X,\ZZ)=\operatorname{Weil}(X)\cong\ZZ$,
 \item $D_i$ are orbismooth curves, intersecting transversally,
 \item $m_i$ are coprime, and each $m_i$ is coprime with the degree of $D_i$.
 \end{enumerate}
\item Conversely, given any $(X,\Delta)$ satisfying \emph{(a),\,(b),\,(c),} 
there is a unique Seifert bundle $M\rightarrow X$ such that $M$ is a rational homology sphere with $H_1(M,\ZZ)=0$.
\end{enumerate}
\end{theorem}

In (1b) above, $D_i$ are axes in a cyclic coordinate chart of the form $\CC^2/\ZZ_m$, that is $D_i$ is either $\{z_1=0\}$ or $\{z_2=0\}$.

There are very few rational homology spheres which admit positive Sasakian structures.

\begin{theorem}[{\cite[Theorem 10.2.19]{BG}, \cite[Theorem 1.4]{K}}]\label{thm:pos-sas} 
Suppose that a rational homology sphere $M$ admits a positive Sasakian structure. 
Then $M$ is spin and $H_2(M,\ZZ)$ is one of the following:
 $$
 0,\,\, \ZZ_m^2,\,\,  \ZZ_5^4,\,\, \ZZ_4^4,\,\,
 \ZZ_3^4,\,\,\ZZ_3^6,\,\, \ZZ_3^8,\,\, \ZZ_2^{2n},\,\, 
 $$
where $n>0$, and $m\geq 2$, $m$ not divisible by $30$. Conversely, all these cases do occur.
\end{theorem}

\begin{theorem}[{\cite[Theorem 10.3.14]{BG}}]\label{thm:pos-sphere} 
Let $M$ be a rational homology sphere. If it admits a Sasakian structure, then it is either positive, 
and the torsion in $H_2(M,\ZZ)$ is restricted by Theorem \ref{thm:pos-sas}, or it is negative. There exist infinitely many such manifolds which admit negative Sasakian structures but no positive Sasakian structures. There exist infinitely many positive rational homology spheres which also admit negative Sasakian structures.
\end{theorem}

In particular, any torsion group 
which is realizable by a simply connected  Sasakian rational homology sphere $M$ but 
which does not appear in the list given by Theorem \ref{thm:pos-sas} gives an example of an answer to Question \ref{quest:bg-main}.

In the smaller class of semi-regular Sasakian structures, Question \ref{quest:bg-main} is answered by the following result. 

\begin{theorem}[\cite{MT}]\label{thm:main2} 
Let $m_i\geq 2$ be pairwise coprime, and $g_i={\frac 12}(d_i-1)(d_i-2)$. 
Assume that $\gcd(m_i,d_i)=1$. Let $M$ be a Smale-Barden manifold with $H_2(M,\ZZ)
=\mathop{\oplus}\limits_{i=1}^r\ZZ_{m_i}^{2g_i}$ and spin, with the exceptions
$\ZZ_m^2, \ZZ_2^{2n}, \ZZ_3^6$. Then $M$ admits a negative semi-regular Sasakian structure. 
Conversely, if $M$ is a simply-connected rational homology $5$-sphere admitting a 
semi-regular Sasakian structure, then it must satisfy the above assumptions.
\end{theorem}

Theorem 26 in \cite{CMST} shows that in the quasi-regular case Theorem \ref{thm:main2} does not hold.

The above discussion motivates the first problem we want to address, which
certainly contributes to Question \ref{quest:bg-main}.

\begin{question} \label{quest:bg-main2}
Which simply connected rational homology spheres admit both negative and positive Sasakian structures?
\end{question}

The following is a partial result on this question.

\begin{theorem}[\cite{G}]\label{thm:gomez} 
Any 
simply connected
rational homology sphere from Koll\'ar's list (Theorem \ref{thm:pos-sas})
except (possibly) $\ZZ_m^2$,  $m < 5$, and $\ZZ_2^{2n}$, 
$n>0$, admits both negative and positive Sasakian structures.
\end{theorem}

We give a complete answer to Question \ref{quest:bg-main2} in Theorem \ref{thm:neg-pos-main}: all positive Sasakian 
simply connected rational homology spheres also  admit negative Sasakian structures.

\medskip

Simply connected rational homology $5$-spheres belong to a wider class of $5$-manifolds.  A $5$-dimensional simply connected manifold $M$ is called a {\it Smale-Barden manifold}. These manifolds are classified by their second homology group over $\ZZ$ and the so-called {\it Barden invariant} \cite{B}, \cite{S}. 
 In more detail,
let $M$ be a compact smooth oriented simply connected $5$-manifold. 
Let us write $H_2(M,\ZZ)$ as a direct sum of cyclic groups of prime  power order
  \begin{equation*} %\label{eqn:H2-1}
  H_2(M,\ZZ)=\ZZ^k\oplus \big( \mathop{\oplus}\limits_{p,i}\, \ZZ_{p^i}^{c(p^i)}\big),
  \end{equation*}
where $k=b_2(M)$. Choose this decomposition in a way that the second Stiefel-Whitney class map
  $w_2: H_2(M,\ZZ)\rightarrow\ZZ_2$
is zero on all but one summand $\ZZ_{2^j}$. The value of $j$ is unique, it
is denoted by $i(M)$ and is called the Barden invariant. The fundamental question arises, 
which Smale-Barden manifolds admit Sasakian structures?

One can ask for a generalization of Question \ref{quest:bg-main} for Smale-Barden manifolds with $b_2>0$.

\begin{question}[{\cite[discussion on p. 359]{BG}}]\label{quest:torsion} 
Determine which torsion groups correspond to Smale-Barden manifolds admitting negative Sasakian structures.
\end{question}  

\begin{theorem}[\cite{G}]\label{thm:gomez2} 
For any pair $(n,s)$, $n>1$, $s>1$ there exists a Smale--{Barden} 
manifold $M$ which admits a negative Sasakian structure such that $H_2(M)_{\mathrm{tors}} =(\ZZ_n)^{2s}$.
\end{theorem}

Note that the manifold $M$ provided in Theorem \ref{thm:gomez2} is not a rational homology sphere.

It is expected that simply connected $5$-manifolds with negative Sasakian structures should be little 
constrained in terms of what kind of $H_2(M,\ZZ)$ is. However, there are only few known examples:
\begin{itemize}
\item circle bundles over the Fermat hypersurfaces of degree $d$ in $\CP^3$, $d\geq 5$, which yield 
{\it regular} negative Sasakian structures on 
  $$
  \#_k(S^2\times S^3), k=(d-2)(d^2-2d+2)+1,
  $$
thus $k$ begins with $k=52$ (see \cite[Example 5.4.1]{BG});
\item  some links yield negative Sasakian structures on 
  $$ 
  \#_7(S^2\times S^3), \#_{12}(S^2\times S^3) , \#_{20}(S^2\times S^3)
  $$
 (\cite[Corollary 10.3.18]{BG} --  note the typo in {\it ibid.} listing $\#_2$ instead of $\#_{20}$);
\item $S^5$ admits infinitely many inequivalent negative Sasakian structures \cite[Proposition 10.3.13]{BG}.
\end{itemize}
Motivated by this, Boyer and Galicki pose the following question that pertains to Smale-Barden manifolds with torsion-free homology.

\begin{question}[{\cite[Open Problem 10.3.6]{BG}}]\label{quest:connected-sum} 
Show that all $\#_k(S^2\times S^3)$ admit negative Sasakian structures 
or determine precisely for which $k$ this holds true.
\end{question}

In Theorem \ref{thm:all_k} we give a complete answer to this question: any $\#_k(S^2\times S^3)$ admits negative Sasakian structures. In this work we consider only spin 5-manifolds $M$. The reason for this can be explained as follows. A Sasakian structure $(\eta,\xi,\phi,g)$ is said to be $\eta$-Einstein if the Ricci curvature tensor of the metric $g$ satisfies the equation
$\operatorname{Ric}_g=\lambda g+\nu \eta\otimes\eta$ for some constants $\lambda$ and $\nu$. It is known \cite{BGM} that by the orbifold version of the theorem of Aubin and Yau every negative Sasakian structure can be deformed to a Sasakian $\eta$-Einstein structure. On the other hand, the following theorem holds.

\begin{theorem}[{\cite[Theorem 14]{BGM}}]\label{thm:w2} 
Let $M$ be a non-spin manifold with $H_1(M,\ZZ)$ torsion free. Then $M$ does not admit a Sasakian $\eta$-Einstein structure.
\end{theorem}

This shows that negative Sasakian Smale-Barden manifolds should be spin (that is, their Barden invariant is $i(M)=0$). This can also
be checked by the proof of \cite[Proposition 2.6]{BGN} that also works for simply connected negative Sasakian manifolds.

Finally, let us mention that negative Sasakian structures are an important tool  in constructing Lorentzian Sasaki-Einstein metrics and, therefore, receive extra attention in physics \cite{BGM}.  

\noindent{\bf Acknowledgment}. The first author was partially supported by Project MINECO (Spain) PGC2018-095448-B-I00. 
The third author was supported by the National Science Center (Poland), grant. no. 2018/31/B/ST1/00053

%%%%%%%%%%%%%%%%%%%%%%%%%%%%%%%%%%%%%%%%%%%%%%%%%%%%%%%%%%%
\section{Sasakian manifolds and Seifert bundles}\label{sec:basic-def}
%%%%%%%%%%%%%%%%%%%%%%%%%%%%%%%%%%%%%%%%%%%%%%%%%%%%%%%%%%%

Let $M$ be a smooth manifold of dimension $2n+1$. A {\it  contact metric structure} on $M$ consists of a quadruplet $(\eta,\xi,\phi, g)$, where $\eta$ is contact form, $\xi$ is the Reeb vector field of this form, $\phi$ is a $C^{\infty}$-section of $\operatorname{End}(TM)$ and $g$ is a Riemannian metric on $M$, satisfying the following conditions:
 $$\phi^2=-\operatorname{Id}+\xi\otimes\eta, \;\;\; g(\phi V,\phi W)=g(V,W)-\eta(V)\eta(W),$$
for any vector fields $V,W$ on $M$.
Given a contact metric structure $(\eta,\xi,\phi,g)$ on $M$, one defines the fundamental 2-form $F$ on $M$ by the formula
 $$F(V,W)=g(\phi V,W).$$ 
One can check that $F(\phi V,\phi W)= F(V,W)$ and that $\eta\wedge F^n\not=0$ everywhere. 
A contact metric structure is {\it K-contact} if $\mathcal{L}_{\xi}g=0$ for the Lie derivative $\mathcal{L}_{\xi}$. 

\begin{definition} 
A contact metric structure $(\eta,\xi,\phi,g)$ on $M$ is called normal, if the Nijenhuis tensor $N_{\phi}$ given by the formula
 $$N_{\phi}(V,W)=\phi^2[V,W]+[\phi V,\phi Y]-\phi[\phi V,W]-\phi[V,\phi W]$$
satisfies the equation
 $$N_{\phi}=-d\eta\otimes\xi.$$
A Sasakian structure is a K-contact structure which is normal.
\end{definition}

\begin{definition} 
A Sasakian structure on a compact manifold $M$ is called {\it quasi-regular} if there is a positive integer $\delta$ satisfying the condition that each point of $M$ has a neighbourhood $U$ such that each leaf of the foliation $\mathcal{F}_{\xi}$ passes through $U$ at most $\delta$ times. If $\delta=1$ the structure is called {\it regular}.
\end{definition}

We freely use the notion of a cyclic orbifold referring to \cite{BG}, \cite{K1}, \cite{M}, \cite{MRT}. Note that in this work we need only $4$-dimensional cyclic orbifolds, so our exposition will be restricted only to this case and simplified accordingly. 
Also we refer for some technical results to  \cite{M}, where the symplectic versions are stated. Note that a K\"ahler orbifold is in particular
an almost-K\"ahler and hence a symplectic orbifold, hence the results of \cite{M} hold also.

For a singular K\"ahler manifold $X$ with cyclic singularities, a local model around a singular point $x\in X$ is of the form 
$\mathbb{C}^2/\ZZ^d$, where $\theta=\exp(2\pi i/d)\in S^1$ acts on the neighbourhood of $x$ by the formula
  \begin{equation}\label{eqn:e1}
  \exp(2\pi i/d)(z_1,z_2)=(e^{2\pi i e_1/d}z_1,e^{2\pi i e_2/d}z_2), %(\theta^{e_1}z_1,\theta^{e_2}z_2),
  \end{equation}
where $\gcd(e_1,d)=\gcd(e_2,d)=1$. We will write $d=d(x)$. 
An orbismooth curve (as in (1b) of Theorem \ref{thm:Ko}) is a complex curve $D\subset X$ 
such that around $x$ it is of the form
$D=\{z_1=0\}$ or $D=\{z_2=0\}$.
Two orbismooth curves $D_1,D_2$ intersect nicely if at every intersection point $x\in D_1\cap D_2$, 
there is a chart $\CC^2/\ZZ_d$ at $x$ such that $D_1=\{(z_1,0)\}$ and $D_2=\{(0,z_2)\}$.

Here is a method of  constructing cyclic K\"ahler orbifolds.
\begin{proposition}[\cite{M}]\label{prop:sympl-orbi} 
Let $X$ be a singular K\"ahler $4$-manifold with cyclic singularities, and set of singular points $P$. 
Let $D_i$ be embedded orbismooth curves intersecting nicely. Take coefficients $m_i>1$ such that $\gcd(m_i,m_j)=1$, if $D_i$ and $D_j$ intersect. Then there exists a K\"ahler orbifold structure $X$ with isotropy surfaces $D_i$ of multiplicities $m_i$, and singular points $x\in P$ of multiplicity $m=d(x)\prod_{i\in I_x}m_i$, where $I_x=\{i\,|\,x\in D_i\}$.
\end{proposition} 

If $P=\emptyset$, then the family $D_i$ with multiplicities $m_i$ as in Proposition \ref{prop:sympl-orbi} defines a structure of a {\it smooth} cyclic K\"ahler orbifold.

An important basic tool of constructing cyclic K\"ahler  orbifolds in this work is by blowing-down complex surfaces along chains of smooth rational curves of negative self-intersection $<-1$
(i.e.\ no $(-1)$-curves). We freely use the theory of complex surfaces referring to \cite{BPV}, \cite{GS}, \cite{H}.

\begin{proposition}[\cite{BPV}] \label{prop:Hirz-Jung}
Consider the action of the cyclic group $\ZZ_m$ on $\CC^2$ given by $(z_1,z_2)\mapsto (\eta z_1,\eta^r z_2)$, where
$\eta=e^{2\pi i /m}$, $0<r<m$ and $\gcd(r,m)=1$.
 Then write a continuous fraction
 $$
  \frac{m}{r}=[b_1,\ldots, b_l]=b_1- \frac{1}{b_2-\frac{1}{b_3- \ldots}}
  $$
The resolution of $\CC^2/\ZZ_m$ has an exceptional divisor formed by a chain of 
smooth rational curves of self-intersection numbers
$-b_1,-b_2,\ldots,-b_l$.
\end{proposition}

\begin{proposition}[{\cite[Lemma 15]{CMST}}] \label{prop:constr-orbi} 
Conversely, let $X$ be a smooth complex surface containing a chain of smooth rational curves $E_1,\ldots, E_l$ of self-intersections
$-b_1,-b_2,\ldots$, $-b_l$, with all $b_i\geq 2$, intersecting transversally (so that $E_i\cap E_{i+1}$ are nodes, $i=1,\ldots, l-1$).
Let $\pi:X\to \bar X$ be the contraction of $E=E_1\cup\ldots \cup E_l$. Then $\bar X$ has a cyclic singularity at $p=\pi(E)$,
with an action given by Proposition \ref{prop:Hirz-Jung}. Moreover, if $D$ is a curve
intersecting transversally a tail of the chain (that is, either $E_1$ or $E_l$ at a non-nodal point),
then the push down curve $\bar D=\pi(D)$ is an orbismooth curve in $\bar X$. 
\end{proposition}

We also need the self-intersection of the orbismooth curve of Proposition \ref{prop:constr-orbi}.
Take $D$ intersecting $E_1$ and $\bar D=\pi(D)$ (the case where $D$ intersects $E_l$
can be treated by reversing the chain of rational curves). Let $[b_1,\ldots, b_l]=\frac{m}{r}$. Then 
 \begin{equation}\label{eqn:barD2}
 \bar D^2= D^2+ \frac{r}{m}\, .
 \end{equation}
We check this as follows. Let $\beta_i=[b_i,\ldots, b_l]$, so that $\beta_i=b_i-\frac1{\beta_{i+1}}$. We prove the assertion 
by induction on $l$.
 Blow-down the chain of  $l-1$ curves $E_2,\ldots, E_l$, getting a blow-down map 
 $$\varpi:X\to \hat X$$ 
 with a singular point $q$. 
 The curve $\hat E_1=\varpi(E_1)$ has self-intersection $\hat E_1^2=-b_1+\frac{1}{\beta_2}=-\beta_1$ by the induction 
 hypothesis while $\hat D = \varphi(D)\cong D$, so $\hat D^2 = D^2$.
 Contracting $\hat E_1$ yields another map
 \[
 \psi: \hat X \to \bar X
 \]
 such that $\pi=\psi\circ \varpi$, and $\bar D = \psi(\hat D)$.
 Take the pull-back $\psi^*(\bar D)=\hat D+x \hat E_1$, and compute $x\in \QQ$ by using 
  $0=\hat E_1\cdot \psi^*(\bar D)= 1 -x\beta_1$, so $x=\frac{1}{\beta_1}$ and $\psi^*(\bar D)=D+\frac{1}{\beta_1} \hat E_1$.
  Now $\bar D^2=\psi^*(\bar D)^2=\hat D\cdot \psi^*(\bar D)= \hat D^2+x = D^2+\frac{1}{\beta_1}$, as stated.

The basic method of constructing Sasakian structures is the method of Seifert bundles \cite{BG}, \cite{K}, \cite{K1}.

\begin{definition} \label{def:Seifert-bdle}
Let $X$ be a cyclic oriented $4$-orbifold. A Seifert bundle over $X$ is an oriented $5$-manifold $M$ endowed with a smooth $S^1$-action and a continuous map $\pi: M\rightarrow X$ such that for an orbifold chart $(U,\tilde{U},\ZZ_m,\varphi)$ there is a commutative diagram
$$
\CD
(S^1\times \tilde{U})/\ZZ_m @>{\cong}>> \pi^{-1}(U)\\
@VVV @V{\pi}VV\\
\tilde U/\ZZ_m @>{\cong}>> U
\endCD
$$
where the action of $\ZZ_m$ on $S^1$ is by multiplication by $\exp (2\pi i/m)$, and the top diffeomorphism is $S^1$-equivariant.
\end{definition}

The relation between quasi-regular Sasakian manifolds and cyclic orbifolds is given by the following result.

\begin{theorem}[{\cite[Theorems 7.5.1 and 7.5.2]{BG}}]\label{thm:sas-orb} 
Let $M$ be a manifold endowed with a quasi-regular Sasakian structure $(\eta,\xi, \phi,g)$. Then the space of leaves of the foliation $\mathcal{F}_{\xi}$  determined by the Reeb vector field has a natural structure of a cyclic K\"ahler orbifold, and the projection $M\rightarrow X$ is a Seifert bundle. Conversely, if $(X,\omega)$ is a cyclic K\"ahler orbifold and $M$ is the total space of the Seifert bundle
 determined by the class $[\omega]$ of the K\"ahler form, then $M$ admits a quasi-regular Sasakian structure.
\end{theorem}

In particular, regular Sasakian structures correspond to circle bundles $M\rightarrow (X,\omega)$ over K\"ahler manifolds with integral K\"ahler form $\omega$. These bundles are determined by the first Chern class equal to $[\omega]$. They are often called the Boothby-Wang fibrations.
Note that by a theorem of Rukimbira \cite{R}, any manifold which admits a Sasakian structure, admits a quasi-regular one. In this work we are interested in the existence questions, and all Sasakian structures constructed in this article will be quasi-regular. Therefore, our approach amounts to constructing Seifert bundles over cyclic K\"ahler orbifolds with prescribed properties. It is also useful to have an intermediate class of Sasakian structures  \cite{K1}, \cite{MRT}, 
\cite{MT}. A Sasakian structure is called {\it semi-regular}, if it is determined by a Seifert bundle $M\rightarrow X$ whose base orbifold is smooth.
  
Let us mention the following. For the convenience of references, we will interchangebly follow the terminology and notation of \cite{BG}, \cite{K}, \cite{K1}. If $X$ is a cyclic orbifold with singular set $P$ and a family of surfaces $D_i$ we will say that we are given a divisor $\mathop{\cup} D_i$ with multiplicities $m_i>1$. The formal sum
  \begin{equation}\label{eqn:Delta}
  \Delta=\sum_i\left(1-{\frac 1m_i}\right)D_i
  \end{equation}
will be called the {\it branch divisor}.

Given a cyclic K\"ahler orbifold $X$, with singular points $P$, and branch divisor (\ref{eqn:Delta}), we need an extra piece of information in order
to determine a Seifert bundle $M\to X$. 
For each point $x\in P$ with multiplicity $m=d m_1m_2$ as in Proposition \ref{prop:sympl-orbi},
we have an adapted chart $U\subset \CC^2/\ZZ_m$ with action 
   \begin{equation}\label{eqn:Zm}
  \exp(2\pi i/m)(z_1,z_2)=(e^{2\pi i j_1/m}z_1,e^{2\pi i j_2/m}z_2),
  \end{equation}
so that the local model of the Seifert fibration is as in Definition \ref{def:Seifert-bdle}. 
We call $j_x=(m,j_1,j_2)$ the {\it local invariants} at $x\in P$. Note that this relates to the original K\"ahler structure of $X$ given
by (\ref{eqn:e1}) by the formulas
$j_1=m_1e_1$ and $j_2=m_2e_2$.
Assume that $D_1=\{(z_1,0)\}$ is one of the isotropy surfaces with multiplicity $m_1$. 
The local invariant of $D_1$ is by definition, $j_{D_1}=(m_1,j_2)$, where $j_2$ is considered modulo $m_1$. 
This also yields a {\it compatibility condition} for the local invariants of singular points and isotropy surfaces.

In order to construct Seifert bundles, we need to assign local invariants to each of the singular points $x\in P$ and each
of the isotropy surfaces $D_i\subset X$, in a compatible way. For this we need to choose for each $D_i$ some $j_i$
with $\gcd(j_i,m_i)=1$. Then we can assign local invariants at the singular points using the following result.
In \cite{M} it is stated in the case that the isotropy surfaces $D_i$ are disjoint, but it is also valid in the singular points
which do not lie in the intersection of two surfaces.

\begin{proposition}[{\cite[Proposition 25]{M}}] \label{prop:25-M}
Suppose that $X$ is a cyclic $4$-orbifold 
such that each singular point $x\in P$ lies in a single isotropy surface $D_i$,
if any.
Take integers $j_i$ with $\gcd(m_i,j_i) = 1$ for each $D_i$.  Then there exist local invariants for $X$.
\end{proposition}

Once compatible local invariants are fixed, a Seifert bundle is determined by its orbifold Chern class.
Given a Seifert bundle $\pi: M\rightarrow X$, the order of a stabilizer (in $S^1$) of any point $p$ in the fiber over $x\in X$ 
is denoted by $m=m(x)$, as in (\ref{eqn:Zm}).

\begin{definition} For a Seifert bundle $M\rightarrow X$ define the first Chern class as follows. Let $l=\operatorname{lcm}(m(x)\,|\,x\in X)$. 
Denote by $M/l$ the quotient of $M$ by $\ZZ_l\subset S^1$. Then $M/l\rightarrow X$ is a circle bundle with the first Chern class $c_1(M/l)\in H^2(X,\ZZ)$. 
Define
$$c_1(M)={\frac 1l}c_1(M/l)\in H^2(X,\mathbb{Q}).$$
\end{definition}

In the proofs of our results we basically need to construct a Seifert bundle $M\rightarrow X$ as in Proposition \ref{prop:c1} such that $H_1(M,\ZZ)=0$.
In order to explain the construction, we recall some results from \cite{M}.
  As before, $X$ is a $4$-dimensional  cyclic orbifold. Let $p_j$ be a cyclic isotropy point of some order $d_j>0$ and $P$ be the set of all such points. Consider small balls around all the points $p_j\in P$, say, $B=\sqcup B_j$. Every $L_j=\partial B_j$ is a lens space of order $d_j$. Let $L=\sqcup L_j$. From the Mayer-Vietoris exact sequence of $X-B$ and $\bar{B}$ and the equalities $H^1(L_j,\ZZ)=0,\,H^2(L_j,\ZZ)=\ZZ_{d_j}$ one derives an exact sequence
$$0\rightarrow H^2(X,\ZZ)\rightarrow H^2(X-P,\ZZ)\rightarrow \mathop{\oplus}_j\ZZ_{d_j}.$$ 
Let $X^o=X-B$. It is a manifold with boundary $L$. There is the Poincar\'e duality
$$H^2(X^o,\ZZ)\times H^2(X^o,L,\ZZ)\rightarrow H^4(X^o,L,\ZZ).$$
There are isomorphisms
$$H^2(X-P,\ZZ)=H^2(X^o,\ZZ)$$
and
$$H^k(X^o,L,\ZZ)=H^k(X,B,\ZZ)=H^k(X,P,\ZZ)=H^k(X,\ZZ),\,k\geq 2,$$
since the dimension of $P$ is zero. Hence the Poincar\'e duality is a perfect pairing
$$H^2(X-P,\ZZ)\times H^2(X,\ZZ)\rightarrow \ZZ.$$
In particular, the class $[D_i]\in H_2(X,\ZZ)$ yields a map $H^2(X-P,\ZZ)\rightarrow\ZZ$, and, therefore, a cohomology class in $H^2(X,\mathbb{Q})$ via the inclusion 
$$H^2(X,\ZZ)\subset H^2(X-P,\ZZ).$$
 
The first orbifold Chern class of the Seifert bundle can be calculated using the formula given by the following result, that we state
only for complex orbifolds.

\begin{proposition}[\cite{M}, Proposition 35]\label{prop:c1} Let $X$ be a cyclic $4$-orbifold with a complex structure 
and $D_i\subset X$ complex curves of $X$ which intersect transversally. 
Let $m_i>1$ such that $\gcd(m_i,m_j)=1$ if $D_i$ and $D_j$ intersect. Suppose that there are given local invariants $(m_i,j_i)$ for each $D_i$ and $j_p$ for every singular point $p$, which are compatible. Choose any $0<b_i<m_i$ such that $j_ib_i\equiv 1 (\operatorname{mod}\,m_i).$   Let ${B}$ be a complex line bundle over $X$. Then there exists a Seifert bundle $M\rightarrow X$ with the given local invariants and the first orbifold  Chern class
  $$c_1(M)=c_1({B})+\sum_i\frac{b_i}{m_i}[D_i].$$
The set of all such Seifert bundles forms a principal homogeneous space under $H^2(X,\ZZ)$, where the action corresponds to changing ${B}$.

Moreover, if $X$ is a K\"ahler cyclic orbifold and $c_1(M)=[\omega]$ for the orbifold K\"ahler form, then $M$ is Sasakian.
\end{proposition}

Note that $[D_i]$ are understood as cohomology classes in $H^2(X,\mathbb{Q})$ according to the explanation before Proposition \ref{prop:c1}.

\begin{proposition}[{\cite[Lemma 34]{M}}]\label{prop:lemma34} 
Assume that we are given a Seifert bundle $M\rightarrow X$ over a cyclic orbifold with the set of isotropy points $P$ and 
branch divisor $\sum (1-{\frac 1{m_i}})D_i$. Let $\mu=\operatorname{lcm}(m_i)$. Then $c_1(M/\mu)=\mu \, c_1(M)$ 
is integral in $H^2(X-P,\ZZ)$.
\end{proposition}

 Recall that an element $a$ in a free abelian group $A$ is called {\it primitive} if it cannot be represented 
 as $a=kb$ with non-trivial $b\in A$, $k\in\mathbb{N}$, $k>1$. 
\begin{theorem}[\cite{M}]\label{thm:H1(M)} Suppose that $\pi: M\rightarrow X$ is a quasi-regular Seifert bundle over a cyclic orbifold $X$ with isotropy surfaces $D_i$ and set of singular points $P$. Let $\mu=\operatorname{lcm}(m_i)$. Then $H_1(M,\ZZ)=0$ if and only if 
\begin{enumerate}
\item $H_1(X,\ZZ)=0,$
\item $H^2(X,\ZZ)\rightarrow \oplus H^2(D_i,\ZZ_{m_i})$ is onto,
\item $c_1(M/\mu)\in H^2(X-P,\ZZ)$ is primitive.
\end{enumerate}
Moreover, $H_2(M,\ZZ)=\ZZ^k\oplus(\, \mathop{\oplus}\limits_i\ZZ_{m_i}^{2g_i})$, $g_i=\text{genus of}\,\,D_i$, $k+1=b_2(X).$
\end{theorem}
Thus, if one wants to check the assumptions of Theorem \ref{thm:H1(M)}, 
one in particular calculates $H^2(X-P,\ZZ)$ and checks the primitivity of $c_1(M/\mu)$ 
in $H^2(X-P,\ZZ)$. The first orbifold Chern class of the Seifert bundle can be calculated using the formula in Proposition \ref{prop:c1}.

\medskip

Let us explain the notion of a definite Sasakian structure in more detail. Recall that a Sasakian structure is positive (negative), if $c_1(\mathcal{F}_{\xi})$ can be represented by a positive (negative) definite $(1,1)$-form. 
The first Chern class $c_1(\mathcal{F}_{\xi})$ is an element of the basic cohomology of the foliated manifold $(M,\mathcal{F}_{\xi})$. A differential form $\alpha$ on $M$ is called {\it basic} if $i_{\xi}\alpha=0$ and $i_{\xi}d\alpha=0$. Consider the spaces of basic $k$-forms $\Omega^k_B(M)$. Clearly, the de Rham differential takes basic forms to basic forms, the restriction of the differential $d$ onto $\Omega^*_B(M)$ yields a differential complex $(\Omega^*_B,d_B)$ and a basic cohomology $H^*_B(\Omega_B^*,d_B)=H_B^*(M)$. In the Sasakian case, the transverse geometry is K\"ahler, and the basic cohomology inherits the bigrading $H^{p,q}_B(M)$.
 If $\pi: M\rightarrow X$ is a Seifert bundle determined by a quasi-regular Sasakian structure, then by \cite[Proposition 7.5.23]{BG},  $c_1(\mathcal{F}_{\xi})=\pi^*c_1^{\orb}(X)$. Since $c_1(\mathcal{F}_{\xi})$ is represented by a basic form (see \cite[Section 7.5.2]{BG}), it follows that $c_1(\mathcal{F}_{\xi})$ is represented by a  negative definite $(1,1)$-form if and only if $c_1^{\orb}(X)<0$ (that is, represented by a negative definite orbifold $(1,1)$-form). Thus, we get the following (see also \cite[Section 4]{BGM}).
 
\begin{proposition}\label{prop:c1neg} A $5$-manifold $M$ admits a quasi-regular negative Sasakian structure if and only if the base $X$ of the corresponding Seifert bundle $M\rightarrow X$ has the property that the canonical class  $K_X^{\orb}$ is ample.
\end{proposition}

In the case of rational homology spheres Sasakian structures fall into two classes.  
\begin{definition} A Sasakian structure is called {\it canonical (anti-canonical)}, if $c_1^{\orb}(X)$ is a positive (negative) multiple of $[d\eta]_B$. 
\end{definition}

Clearly, an anti-canonical Sasakian structure is negative.
 Note that in general there are positive or negative Sasakian structures on $M$  which are not canonical or anti-canonical. If $b_2(M)>0$, then the condition of being canonical or anti-canonical chooses a ray in the K\"ahler cone ${\mathcal K}(X)$. 
 Thus, there are many positive (negative) Sasakian structures determined by Seifert bundles $M\rightarrow X$ whose first Chern class $c_1(M)$ is not proportional to $c_1(\mathcal{F}_{\xi})$. However, if $M$ is a rational homology sphere, then any Sasakian structure is either canonical or anti-canonical.

\begin{proposition}[{\cite[Proposition 7.5.29]{BG}}]\label{prop:can-sphere} Let $M$ be a $(2n+1)$-dimensional rational homology sphere. Any Sasakian structure on $M$ satisfies $c_1(\mathcal{F}_{\xi})=a[d\eta]_B$ for some non-zero constant $a$. Hence, any Sasakian structure on $M$ is either positive or negative.
\end{proposition}

Finally, recall that we are interested on Smale-Barden manifolds obtained as particular Seifert bundles $M\rightarrow X$. Therefore, we need to check that $\pi_1(M)=1$. We will systematically use the following. 
\begin{definition} The orbifold fundamental group $\pi_1^{\orb}(X)$ is defined as 
   $$
   \pi_1^{\orb}(X)=\pi_1(X-(\Delta\cup P))/\langle \gamma_i^{m_i}=1\rangle,
   $$
where $\langle\gamma_i^{m_i}=1\rangle$ denotes the following relation on $\pi_1(X-(\Delta\cup P))$: for any small loop $\gamma_i$ around a surface $D_i$ in the branch divisor, one has $\gamma_i^{m_i}=1$.
\end{definition}

We will use without further notice the following exact sequence
  $$
  \cdots\rightarrow \pi_1(S^1)=\ZZ\rightarrow\pi_1(M)\rightarrow\pi_1^{\orb}(X)\rightarrow 1.
  $$
It can be found in \cite[Theorem 4.3.18]{BG}. It is easy to see that if $H_1(M,\ZZ)=0$ and $\pi_1^{\orb}(X)$ is abelian, then  
$\pi_1(M)$ must be trivial.  This holds since if $H_1(M,\ZZ)=0$, then $\pi_1(M)$ has no abelian quotients. As $\pi_1^{\orb}(X)$ is assumed abelian,
we find that $\pi_1(M)$ is a quotient of $\ZZ$, hence again abelian. This implies that $\pi_1(M)=0$.

The following result will be used in the calculations of the orbifold fundamental groups.
\begin{proposition}[\cite{N}]\label{prop:nori}  If $Z$ is a smooth simply-connected projective surface with smooth complex curves $C_i$ intersecting transversally and satisfying $C_i^2>0$, then $\pi_1(Z-(C_1\cup \ldots\cup C_r))$
is abelian.
\end{proposition}

%%%%%%%%%%%%%%%%%%%%%%%%%%%%%%%%%%%
\section{Rational homology spheres with positive and negative Sasakian structures}
%%%%%%%%%%%%%%%%%%%%%%%%%%%%%%%%%%%

The aim of this section is to prove the following result.

\begin{theorem}\label{thm:neg-pos-main} 
Any simply connected rational homology sphere admitting a positive Sasakian structure  admits also a negative Sasakian structure.
\end{theorem}

By Theorem \ref{thm:gomez}, it only remains to check that the homology groups
$H_2(M,{\mathbb Z})= {\mathbb Z}_2^{2n}$, $n>0$, and ${\mathbb Z}_m^2$, $m< 5$, 
from Theorem \ref{thm:pos-sas} can be covered.
We shall work out each of them separately.

We begin with some known facts on Hirzebruch surfaces \cite{BPV}, \cite{GS}, \cite{H}. 
Fix $n\geq 0$. By definition, the Hirzebruch surface $\mathbb{F}_n$  is the projectivization of the vector bundle
 $$
 \mathcal{O}_{\CP^1}(n)\oplus\mathcal{O}_{\CP^1}.
 $$ 
For a holomorphic section
 $\sigma: \CP^1\rightarrow\mathcal{O}_{\CP^1}(n)$, we
denote by $E_\sigma$ the image of $(\sigma,1): \CP^1\rightarrow \mathcal{O}_{\CP^1}(n)\oplus\mathcal{O}_{\CP^1}$
in $\mathbb{F}_n$. 
The curve $E_0\subset\mathbb{F}_n$ is called the zero section of $\mathbb{F}_n$. As $E_0\equiv E_\sigma$ for all sections $\sigma$,
we have that $E_0^2=n$. Let $C$ denote the fiber of the fibration $\mathbb{F}_n\rightarrow \CP^1$. Then
 $$
 C^2=0, E_0\cdot C=1.
 $$
It is known \cite{T} that for $n>0$ the surface $\mathbb{F}_n$ contains a unique irreducible curve $E_{\infty}$ of negative self-intersection:
 \begin{equation*}%\label{eqn:(2)}
 E_{\infty}\cdot E_{\infty}=-n.
 \end{equation*}
This curve is called the section at infinity since it is  given by the image of 
$(\sigma,0): \CP^1\rightarrow \mathcal{O}_{\CP^1}(n)\oplus\mathcal{O}_{\CP^1}$. One can calculate
 $$
  E_{\infty} \equiv E_0-nC, \quad K_{\mathbb{F}_n} \equiv (n-2)C-2E_0,
  $$
where $K_{\mathbb{F}_n}$ is the canonical divisor. Clearly
 $$
 H^2(\mathbb{F}_n,\ZZ)=\ZZ\langle C,E_{\infty}\rangle.
 $$
 
Now we are ready to prove our first main result.

\begin{proposition}\label{prop:1-neg-pos-main} 
The simply connected rational homology sphere with $H_2(M,\ZZ)=\ZZ_2^{2n}$, $n\geq 1$, and spin
admits a negative Sasakian structure.
\end{proposition}

\begin{proof} 
Consider $\mathbb{F}_n$ with the zero section $E_0$, the section at infinity $E_{\infty}$, 
and the fiber $C$ of $\mathbb{F}_n\rightarrow \CP^1$. Let $\beta\geq 1$ and  take a smooth divisor 
 $$
 D\equiv C+\beta E_0.
 $$ 
This exists because $ C+\beta E_0$ is  very ample (it can be represented by a K\"ahler form) and Bertini's theorem is applicable.
 By the genus formula $D^2+K_{\mathbb{F}_n} \cdot D =2g(D)-2$, we get
 \begin{equation}\label{eqn:(2)}
 g(D)={\frac{\beta^2-\beta}2}\cdot n. 
 \end{equation}
One also calculate
 $$
 D\cdot E_{\infty}=(C+\beta E_0)\cdot E_{\infty}= 1,
 $$
as $E_0\cdot E_\infty=0$. Consider also rational curves $D_i=E_{\sigma_i}$, $i=1,\ldots, s$, which are
sections, $D_i\equiv E_0$. These can be taken to intersect transversally (each other and also to $D$).

Let $X$ be the orbifold obtained by the blow-down of $\mathbb{F}_n$ along $E_{\infty}$. 
By Proposition \ref{prop:constr-orbi}, 
$X$ is a cyclic orbifold with a cyclic singularity $p$ of degree $d(p)=n$. 
For simplicity we denote curves on $\mathbb F_n$ and on $X$ by the same letters.
Endow $X$ with an orbifold structure with the branch divisor
 \begin{equation}\label{eqn:(3)}
 \Delta= \left(1-\frac{1}{m}\right) D+ \sum_{i=1}^s \left(1-\frac{1}{m_i}\right) D_i \, ,
  \end{equation}
where $m_1,\ldots, m_s$ are pairwise coprime, and the $m_i$ are coprime to $m$ and $n$ (but possibly
$\gcd(n,m)>1$).

One can easily calculate that
 \begin{align}
 \label{eqn:H_2}
  & H_2(X,\ZZ) = H^2(X-\{p\},\ZZ) = \ZZ\langle C\rangle, \\
  \label{eqn:dual}
  & H_2(X-\{p\},\ZZ) = H^2(X,\ZZ) = \ZZ\langle E_0\rangle.
 \end{align}
Note that $C$ is a divisor passing through the singular point $p$.
Since the (co)homology groups in \eqref{eqn:H_2}, \eqref{eqn:dual} are dual, we can calculate the intersection numbers over $\mathbb{Q}$ and obtain
 $$
 C^2={\frac 1n}
 $$
in $H_2(X,\mathbb{Q})$. 
Indeed,
 we can write $E_0\equiv a\, C$, for some $a\in\ZZ$, and substituting this to $n=E_0^2=a\,E_0\cdot C=a$, we get 
$E_0 \equiv n\,C$. Then $C^2=\frac{1}{n^2} E_0^2=\frac1n$. 
The same follows from   (\ref{eqn:barD2}).
We also have
 $$
 D \equiv C+\beta E_0\equiv C+\beta nC=(1+\beta n)C.
 $$

Considering $X$ as an algebraic variety, we calculate the canonical divisor $K_X$ from the adjunction formula (noting that $C$ is
orbismooth):
 $$
 K_X\cdot C+C^2=-\chi_{\orb}(C) = 2g(C)-2 + \left(1-\frac{1}{n}\right).
 $$
Since $g(C)=0$ we get $K_X\cdot C+{\frac 1n}=-2+(1-{\frac 1n})$.
Writing $K_X=b\, C$, $b\in\mathbb{Q}$, and substituting in the above, we get $b=-(n+2)$. So finally
  $$
  K_X=-(n+2)C.
  $$
Since all $D_i$ are homologous to $E_0$, we obtain the following formula for the orbifold canonical class:
\begin{equation}\label{eqn:(5)}
  \begin{aligned}
  K_X^{\orb}& =K_X+\left(1-{\frac 1m}\right)D+\sum_{i=1}^s\left(1-{\frac 1{m_i}}\right)D_i \\ 
  &= -(n+2)C+\left(1-{\frac 1m}\right)(1+\beta n)C+\sum_{i=1}^s\left(1-{\frac 1{m_i}}\right)nC \\
 &=\left(-(n+2)+\left(1-{\frac 1m}\right)(1+\beta n)+n\sum_{i=1}^s\left(1-{\frac1{m_i}}\right)\right)C. 
 \end{aligned}
 \end{equation}

Now consider a Seifert bundle 
 $$
  M\rightarrow X
  $$
determined by the orbifold data (\ref{eqn:(3)}). We need to choose local invariants $0<j_i<m_i$, $\gcd(j_i,m_i)=1$,
for each $D_i$, and $0<j<m$, 
$\gcd(j,m)=1$, for $D$. These will be fixed later. The local invariant $j_p$ at the singular point is given by
applying Proposition \ref{prop:25-M}.

By Proposition \ref{prop:c1neg}, we know that the existence of a negative 
Sasakian structure on $M$ is equivalent to the ampleness of the orbifold canonical bundle $K_X^{\orb}$, 
and this is equivalent to the positivity of the coefficient in (\ref{eqn:(5)}). Now we specify some coefficients
in order to realize the desired cohomology groups.
Put $m=2$ and $\beta=2$, so by (\ref{eqn:(2)}), we have $g(D)=n$. By Theorem \ref{thm:H1(M)}, this gives $H_2(M,\ZZ)=\ZZ_2^{2n}$.
Calculating the coefficient in (\ref{eqn:(5)}), %with $\beta=2$, $m=2$, 
we get
  $$
  -{\frac 32}+n\sum_{i=1}^s \left(1-{\frac 1{m_i}}\right).
  $$
Choosing $m_i$ and $s$ large, we can get positivity for any $n$. As we said, we take $m, m_i$ all of them pairwise coprime, and
also $\gcd(m_i,n)=1$ for all $i$.

Note that we have done this calculation assuming that the assumptions of Theorem \ref{thm:H1(M)} are satisfied. 
The assumption (2) holds for $D_i$ since $H^2(X,\ZZ)=\ZZ\la E_0\ra$, $D_i\equiv E_0$, and $\gcd(m_i,n)=1$.
It holds for $D$ since $D\equiv (1+n\beta)C$, so the map $H^2(X,\ZZ)\to H^2(D,\ZZ_m)$ is $[E_0]\mapsto 1+n\beta$, and
$\gcd(1+n\beta,m)=1$ for $\beta=m=2$.

Let us now check assumption (3) of Theorem  \ref{thm:H1(M)}.
By Proposition \ref{prop:c1},
  $$
  c_1(M)=c_1({B})+{\frac bm}[C+\beta E_0]+\sum_i\frac{b_i}{m_i}[E_0],
  $$
for a line bundle $B$, so $c_1(B)=q[E_0]$, for some integer $q\in \ZZ$.
As we take all $m,m_1,\ldots, m_s$ pairwise coprime, we have 
$\mu=m\cdot m_1\cdots m_s$. %By Proposition \ref{prop:lemma34}, 
We need to check that we can arrange for $\mu\,c_1(M)$ to be integral and primitive in $H^2(X-P,\ZZ)$, 
which means that we need
 $$
  c_1(M/\mu)=\mu\,c_1(M)=[C].
 $$
Note also that as $c_1(M)>0$, we have by Proposition \ref{prop:c1} that $M$ is Sasakian.
We compute
 $$ 
 \mu\, c_1(M)=b\cdot m_1\cdots m_s[C+\beta E_0]+\sum_ib_i\cdot m\cdot m_1\cdots\hat{m_i}\cdots m_s[E_0]+\mu\,c_1({B}).
 $$
In $H^2(X,\ZZ)$ we have $[E_0]=n[C]$, so $c_1(M/\mu)=k[C]$ with 
 \begin{equation}\label{eqn:k=1}
 k=(1+{n\beta})bm_1\cdots m_s+n\left(\sum_i b_i mm_1\cdots\hat{m_i}\cdots m_s\right)+ \mu q.
   \end{equation}
Since we already fixed $m=\beta=2$,
we need to arrange $b,b_i,m_i,q$ to get $k=1$. First note that
 \begin{equation}\label{eqn:gcd}
  \gcd((1+{n\beta})m_1\cdots m_s, n mm_1\cdots\hat{m_i}\cdots m_s) =1,
   \end{equation}
by the choice of $m,m_i$. Then there are $\bar b,\bar b_i\in\ZZ$ such that
 $$
 1=(1+{n\beta})\bar b m_1\cdots m_s+n\left(\sum_i \bar b_i mm_1\cdots\hat{m_i}\cdots m_s\right).
 $$
Take $\bar b=b +m q_0$, $\bar b_i=b_i+m_i q_i$, where $0\leq b<m$ and $0\leq b_i<m_i$. Then
 $$
 1=(1+{n\beta}) b m_1\cdots m_s+n\left(\sum_i b_i mm_1\cdots\hat{m_i}\cdots m_s\right) + \left((1+n\beta)q_0+\sum n q_i\right)\mu,
 $$
as required. Finally note that $\gcd(b,m)=1$ and $\gcd(b_i,m_i)=1$, so that we can take local invariants 
$j,j_i$ with $bj\equiv 1 \pmod{m}$ and
$b_ij_i\equiv 1\pmod{m_i}$ and apply Proposition \ref{prop:25-M}.

It remains to show that $\pi_1(M)=1$.  By Proposition \ref{prop:nori}  applied to $\pi_1({\mathbb F}_n-S)$,  $S=D\cup D_1\cup \ldots \cup D_s$, we get that this group is abelian.
Now take the curve $E_\infty$, that intersects transversally $S$ at one point. Take a loop $\delta$ around $E_\infty$, then
$\pi_1({\mathbb F}_n-(S\cup E_\infty))$ is generated by $\delta$ and $\pi_1({\mathbb F}_n-S)$.
Let $\gamma,\gamma_1,\ldots, \gamma_s$ be loops around each of the divisors $D,D_1,\ldots D_s$. 
Take a general fiber $C$ of
${\mathbb F}_n \to \CP^1$, 
and intersecting with $S\cup E_\infty$, we get a homotopy 
$\delta=\gamma^\beta \prod_{i=1}^s\gamma_i$,
since $\beta =C \cdot D$ and $1=C\cdot D_i$, $i=1,\ldots, s$. 
Note also that $\delta,\gamma$ commute since the curves $E_\infty,D$ intersect transversally at one point. 
Also $\gamma_i,\gamma_j$ commute since $D_i,D_j$ intersect transversally (just take the link of the complement of
the two curves around an intersection point, which is a $2$-torus containing both loops; hence they commute in this $T^2$).
Analogously $\gamma_i,\gamma$ commute. Finally, $\delta=\gamma^\beta \prod_{j=1}^s\gamma_j$, so it also
commutes with $\gamma_i$ in $\pi_1({\mathbb F}_n-(S\cup E_\infty))$.

Now recall that $\pi_1^{\orb}(X)$ is the quotient of $\pi_1(X-S)=\pi_1({\mathbb F}_n-(S\cup E_\infty))$ with
the relations $\gamma^m=1$, $\gamma_i^{m_i}=1$. Then we have that
$\pi_1^{\orb}(X)$ is also abelian. It follows that $\pi_1(M)$ is abelian. Since $H_1(M,{\mathbb Z})=0$ we get necessarily $\pi_1(M)=1$. 
By Theorem \ref{thm:H1(M)},  $M$ is a rational homology sphere with $H_2(M,\ZZ)=\ZZ_2^{2n}$.
\end{proof}

The case ${\mathbb Z}_m^2$, $m< 5$ (that is, $m=2,3,4$), is covered by the following result.

\begin{proposition}\label{prop:2-neg-pos-main} 
The simply connected rational homology sphere with $H_2(M,\ZZ)=\ZZ_m^{2}$, $\gcd(m,7)=1$, and spin,
admits a negative Sasakian structure.
\end{proposition}

\begin{proof}
Consider a smooth cubic $C\subset\CP^2$ and a line $L\subset\CP^2$, 
tangent to the cubic at some point $p_1\in L\cap C$
and intersecting transversally at another point $p_2\in L\cap C$. We blow-up at $p_1$, and let $\tilde L$, $\tilde C$
denote the proper transforms. 
Since the initial curves were tangent (intersecting with multiplicity two), 
$\tilde L$ and $\tilde C$ intersect transversally at a point
$\tilde p$. Blowing up again at $\tilde p$, we get two curves $\hat L$ and $\hat C$ in 
$\CP^2\#2\overline{\CP}^2$ with $\hat L\cap \hat C=\{p_2\}$. 
The second blow-up yields a $(-2)$-curve $B$ (obtained by a blow-up of the exceptional divisor) and a new exceptional divisor $E_1$. Since $E_1$ and $\tilde L$ intersect, one can blow up a third time and get new curves $\check{L}, \check{E}{}_1=A, B,E$, 
in 
 $$X=\CP^2\#3\overline{\CP}^2,$$
 where $E$ is the last exceptional divisor. Note that now $\check{L}{}^2=-2$,
and $A\cup B$ is a chain of two $(-2)$-curves. 
Contracting $A\cup B$ and $\check{L}$ and applying Proposition \ref{prop:constr-orbi} we get a cyclic orbifold $\bar X$. Let
 $$\varpi:X\to \bar X$$ 
denote the blow-down map. Let $q_1=\varpi(A\cup B)$ and $q_2=\varpi(\check{L})$ be the two
singular points $P=\{q_1,q_2\}$ of orders $d(q_1)=3$ and $d(q_2)=2$. 

We have that $\bar C=\varpi(\hat C)$ is a genus $1$ orbismooth curve with two singular points
$q_1,q_2$. Let also $\bar E=\varpi(E)$, which is a genus $0$ orbismooth
curve through $q_1,q_2$, although the intersection $\bar C\cap \bar E$ is not nice at $q_1$.

The cohomology of the blow-up is 
$H^2(X,\ZZ)=\ZZ\langle L, A,B,E\rangle$. As we contract $3$ curves, we have
$H^2(\bar X,\QQ)=\QQ\la \bar E \ra$. Moreover $\pi_1(X)=1$, hence $H^2(\bar X,\ZZ)$ is torsion-free.
We compute the intersection numbers using (\ref{eqn:barD2}) to get
 \begin{align*}
  \bar C^2 &=  \hat{C}^2+\frac12+\frac23=7 + \frac76=\frac{49}{6} ,\\
  \bar E^2 &= E^2 + \frac12+\frac23= -1+\frac76=\frac16,\\
  \bar E\cdot \bar C &= E\cdot C+\frac12+\frac23= 0+\frac76 =\frac76,
  \end{align*}
hence $\bar C=7\bar E$. 
It is clear that the cycle $x=\bar C-\bar E$ 
can be pushed-off $P$ (since $x\cdot A=x\cdot B= x\cdot \check{L}=0$), so $x\in H_2(\bar X-P,\ZZ)$. Since $x=6\bar E$ 
and $\bar E\cdot x=1$, both $x,\bar E$ are generators of their respective groups:
 \begin{align*}
 H_2(\bar X,\ZZ) &= H^2(\bar X-P,\ZZ)= \ZZ\la \bar E\ra,\\
 H^2(\bar X,\ZZ) &= H_2(\bar X-P,\ZZ)= \ZZ\la x\ra.
 \end{align*}
For a generic line $L'$ of $\CP^2$, we have 
$L'=\check{L}+B+2A+3E$,
so $\bar L'=\varpi(L')\equiv 3\bar E$.

Take a collection of generic lines $L_1,\ldots, L_s$ in $\CP^2$ that we transfer to $\bar X$.
Endow $\bar X$ with the structure of a cyclic K\"ahler orbifold determined by the set of singular points 
$P=\{p_1,p_2\}$ and a branch divisor
  $$
  \Delta=\left(1-{\frac 1m}\right)\bar C + \sum_{i=1}^s \left(1-\frac{1}{m_i}\right) L_i,
  $$
where $m_i$ are integers pairwise coprime and coprime to $m$.
We construct a Seifert bundle $M\rightarrow \bar X$ determined by local data associated to this
branch divisor. We start by arranging the Chern class of the Seifert bundle
 $c_1(M)=c_1(B)+{\frac bm}[\bar C] +\sum \frac{b_i}{m_i} L_i$, 
 where $c_1(B)=q[ \bar E]$ for some integer $q$. 
Then $\mu=mm_1\cdots m_s$ and
\begin{equation}\label{eqn:c1Mm}
  \begin{aligned}
  c_1(M/\mu) &= m m_1\cdots m_s c_1(B)+bm_1\cdots m_s [\bar C] + \sum
  b_i m m_1\cdots \hat{m}_i\cdots m_s [L'] \\ 
   &= \Big(mm_1\cdots m_s q+7bm_1\cdots m_s +\sum 3 b_i m m_1\cdots \hat{m}_i\cdots m_s \Big) [\bar E].
  \end{aligned}
  \end{equation}
As we are assuming $\gcd(m,7)=1$, there is
 $$
 \gcd(mm_1\cdots m_s, 7 m_1\cdots m_s , 3m m_1\cdots \hat{m}_i\cdots m_s )=1,
 $$
where we also choose $m_i$ coprime to $7$ and to $3$.
So we can solve (\ref{eqn:c1Mm}) to have $c_1(M/\mu)=[\bar E]$. We can arrange that $0<b<m$, $0<b_i<m_i$
and $\gcd(b,m)=1$, $\gcd(b_i,m_i)=1$.
This implies first that $c_1(M)>0$ hence $M$ is Sasakian,
and second that $c_1(M/m) \in H^2(\bar X-P,\ZZ)$ is primitive.
To check the assumptions of Theorem \ref{thm:H1(M)}, we need to
 show the surjectivity of the natural map $H^2(\bar X,\ZZ)\rightarrow H^2(\bar C,\ZZ_m)$. But
 this sends $x\mapsto x\cdot \bar C=7$, so the map is surjective because $\gcd(m,7)=1$.
Also the maps $H^2(\bar X,\ZZ)\rightarrow H^2(\bar L_i,\ZZ_{m_i})$ are surjective
(since $\gcd(m_i,3)=1$),
and so is the map to the direct sum of these.

Now let us compute $K_{\bar X}^{\orb}$. We start by writing $K_{\bar X}=b [\bar E]$ and computing 
$K_{\bar X}\cdot \bar C+\bar C^2=-\chi_{\orb}(\bar C)$. This gives
 $$
  b \frac76 +\frac{49}6 = 1-\frac12 + 1-\frac13  \implies b= -6.
  $$
Therefore 
 $$
K_{\bar X}^{\orb}= \left(-6 + 7 \Big(1-\frac1m\Big) + \sum 3 \Big(1-\frac{1}{m_i}\Big) \right) [\bar E]>0,
 $$
for $s$ and $m_i$ large enough (e.g.\ $s\geq 2$). So $M$ has a negative Sasakian structure.

Finally, we have to calculate $\pi_1^{\orb}(\bar X)$. First we compute $\pi_1(\bar X-P)$, which is generated
by loops $\alpha_1,\alpha_2$ around the points $q_1,q_2$. Note that $\alpha_1^3=1,\alpha_2^2=1$. The
sphere $E$ going through $q_1,q_2$ gives a homotopy $\alpha_1=\alpha_2$. This implies that both
$\alpha_1=\alpha_2=1$, and hence $\pi_1(\bar X-P)=\pi_1(X-(A\cup B\cup \check{L}))=1$. We
remove $\check{C}\cup L_1\cup \ldots\cup L_s$, hence 
 \begin{equation}\label{eqn:pi1XX}
 \pi_1(X-(A\cup B\cup \check{L}\cup \check{C}\cup L_1\cup \ldots\cup L_s)) =\pi_1(\bar X-(\bar{C} \cup L_1\cup
 \ldots \cup L_s))
 \end{equation}
is generated by loops $\gamma,\gamma_1,\ldots, \gamma_s$, where $\gamma$ surrounds $\check{C}$
 and $\gamma_i$ surrounds $L_i$. As all these curves intersect transversally, we have that all these loops
 commute. Therefore (\ref{eqn:pi1XX}) is abelian. This implies that $\pi_1^{\orb}(\bar X)$ is abelian, and
 hence $M$ is simply-connected. Again, by Theorem \ref{thm:H1(M)}, $M$ is a rational homology sphere with $H_2(M,\ZZ)=\ZZ_m^{2}$.
\end{proof}

\begin{remark} In recent work \cite{PW} a complete classification of Sasaki-Einstein rational homology spheres was obtained. This result is a positive type counterpart to the present work which deals with the negative case. One can find more results on the positive case in \cite{BN}.
\end{remark}

%%%%%%%%%%%%%%%%%%%%%%%%%%%%%%%%%%%
\section{Quasi-regular negative Sasakian structures on $\#_k(S^2\times S^3)$}\label{sect:quasi-regular-torsionfree}
%%%%%%%%%%%%%%%%%%%%%%%%%%%%%%%%%%%

In this section we give a complete answer to  Question \ref{quest:connected-sum}:

\begin{theorem} \label{thm:all_k}
Any $\#_k(S^2\times S^3)$
admits a quasi-regular negative Sasakian structure.
\end{theorem}
The proof of this result will follow from Propositions \ref{prop:large_k}, \ref{prop:k=1} and \ref{prop:k=2,3}.

In the proof of Proposition \ref{prop:large_k}
(for most part of Theorem \ref{thm:all_k}) we will use elliptic surfaces with section.
For background, the reader is referred to  \cite{S-Sh}. 
Let 
$$f: Y\rightarrow \CC P^1
$$
 be an elliptic surface with section ${O}$. 
Denote the general fiber by $F$.
The classification of singular fibers goes back to Kodaira,
but we will only need the semi-stable fibers, denoted by $I_n$, $n\in\NN$.
These are either a nodal cubic ($n=1$)
or cycles of $n\geq 2$ smooth rational curves intersecting transversally.
We will number the fiber components cyclically $\Theta_0,\Theta_1,\hdots,\Theta_{n-1}$
such that $\Theta_0$ meets $O$,
and among fiber components, $\Theta_i$ intersects exactly $\Theta_{i-1}$ and $\Theta_{i+1}$,
indices taken modulo $n$.

The invariants of $Y$ are determined by the  Euler characteristic $\chi(Y)=12N$.
Namely we have
\[
q=0,\;\;\; p_g=N-1, \;\;\; b_2 = 12N-2, \;\;\; \text{ and } \;\; h^{1,1}=10N.
\]
Note that the components of a fiber of type $I_n$ together with the zero section
generate a hyperbolic sublattice $L$ of the N\'eron-Severi group NS$(Y)$, of rank $n+1$.
Hence Lefschetz' theorem on $(1,1)$-classes predicts that $n\leq h^{1,1}-1=10N-1$. 
Let $\rho$ denote the Picard number of $Y$.

\begin{proposition} \label{prop:ell_surf}
Let $N>1$. For any $n\in\NN$ such that $n\leq 10N-1$,
there is an elliptic surface $Y\to\CC P^1$ 
with section and $\chi(Y)=12N$ such that there is a singular fiber of type $I_n$
and $\rho=n+1$.
\end{proposition}

\begin{remark}
By the theory of Mordell-Weil lattices \cite{MWL}, the assumption on the Picard number implies 
that there are no other reducible fibers and only sections of finite order.
\end{remark}

\begin{proof}[Proof of Proposition \ref{prop:ell_surf}]
By \cite{Miranda}, the elliptic surfaces with $\chi(Y)=12N$ lie
inside a $(10N-2)$-dimensional moduli space $\mathcal M_N$.
For $N=1$, any such $Y$ is rational with $\rho=10$.
For $N>1$, however, a very general $Y$ has $\rho=2$.
This follows from \cite{Kloosterman}
where all irreducible components of the higher Noether-Lefschetz loci
\[
\mathrm{NL}_r =
\{Y\in\mathcal M_N\; | \; \rho(Y)\geq r\}, \;\;\; 2\leq r\leq 10N,
\]
are shown to have dimension $\geq 10N-r$,
with equality attained outside the isotrivial locus.
But an isotrivial fibration admits only additive singular fibers
while a general surface in $\mathcal M_N$ has only fibers of type $I_1$,
so the claim follows.

Along the same lines, any elliptic surface $Y\in\mathcal M_N$
with a singular fiber of type $I_n$ lies in the non-isotrivial locus
of an irreducible component 
\[
Z_{n+1}\subset \mathrm{NL}_{n+1}.
\]
As soon as this component is shown to be non-empty,
it follows from the above discussion that a very general member will have $\rho=n+1$.
The non-emptiness can be proved in a way similar to the argument in \cite[Thm.\ 8.39]{MWL}:
the elliptic surface $Y$ with $I_n$ fiber can be degenerated to one with $I_{n+1}$
and so on all the way to a terminal object with $I_{10N-1}$
(the maximum possible by Lefschetz as discussed before).
The existence of this terminal object (for any $N\in\NN$) follows from work of Davenport \cite{Dav}
and  Stothers \cite{St}.
In turn, this implies that all the intermediate strata of elliptic surfaces with an $I_n$ fiber,
$n<10N-1$, are non-empty as well.
(Compare the deformation style argument in \cite[Lemma 2.4]{MP}).
\end{proof}

We are now in the position to return to Theorem \ref{thm:all_k}.
We start by the case of large $k$.

\begin{proposition} \label{prop:large_k}
Any $\#_k(S^2\times S^3)$, $k\geq 4$,
admits a quasi-regular negative Sasakian structure.
\end{proposition}

\begin{proof}
Consider first the case $k>1$.
Fix $N>1$ and $n\leq 10N-1$ and 
let $Y$ be an elliptic surface as in Proposition \ref{prop:ell_surf}.
Then $Y$ is simply connected and we have $O^2=-\chi(\mathcal O_Y) = -N$.
Thus $O$ together with the fibre components $\Theta_0,\hdots,\Theta_{n-2}$
form a chain $D$ of $n$ smooth rational curves where all the $\Theta_i$ are $(-2)$-curves.

By Proposition \ref{prop:constr-orbi}, we can contract this chain $D$ to a  point $p$ of order  $d=(N-1)n+1$
on a singular surface $X$. Note that
 $$
 [N,\underbrace{2,\ldots,2}_{(n-1) \text{ times}}]= \frac dn = \frac{(N-1)n+1}{n}\, .
 $$
By construction, we have 
\[
b_2(X) = 12N-2-n, \;\;\; \text{ and } \;\;\;
\rho(X)=1, \;\;\; \text{ where } \;\;\; \mathrm{NS}(X) = \mathbb Z \la F\ra.
\]
Since the fiber $F$ contains $p$, we have by (\ref{eqn:barD2}) that  $F^2=\frac{n}{d}$. 
We continue to prove that $K_X^{\orb}>0$.
To this end, it suffices to prove that $K_X^{\orb}\cdot F>0$, since $\rho(X)=1$.
But the latter can be computed using adjunction through
\[
K_X^{\orb}\cdot F + F^2 = -\chi_{\orb}(F) = 1-\frac 1d,
\]
since $F$ is orbismooth. Therefore
 $$
 K_X^{\orb}\cdot F= \frac{d-n-1}{d}=\frac{(N-2)n}{(N-1)n+1}\, ,
 $$
so the claim follows as soon as $N>2$.
For such $N$, and $1\leq n\leq 10N-1$, we have
$b_2(X)=k+1=12N-2-n \in [2N-1 ,12N-3]$. For $N\geq 3$ this covers
all $k+1\geq 5$, that is $k\geq 4$.

We construct a Seifert bundle $M\rightarrow X$ according to the general procedure as in Proposition \ref{prop:c1}. Note that there is no branch divisor and a singular point of order $d=(N-1)n+1$. We need to ensure that the assumptions of Theorem \ref{thm:H1(M)} are satisfied. 
In this case, we only need to check that $c_1(M)$ is primitive in $H^2(X-P,\mathbb{Z})$. We just take a line bundle $M\rightarrow X-P$, whose first Chern class is ample and primitive, and extend it over $X$ as a Seifert bundle.

Finally we claim that $\pi_1^{\orb}(X)=1$.
To see this, note that $\pi_1^{\orb}(X)$ is a quotient of $\pi_1(X- \{p\}) = \pi_1(Y-D)$.
Here the last group is generated by a small loop $\beta$ around $D$,
so $\pi_1^{\orb}(X)$ is
cyclic generated by $\beta$, and hence abelian, so the claim follows as before. 
Hence we have a quasi-regular negative Sasakian structures on $M$, which is diffeomorphic to  $\#_k(S^2\times S^3)$, by 
the classification of Smale-Barden manifolds and Theorem \ref{thm:w2}. By construction, this can be done  
for all $k\geq 4$.
\end{proof}

In order to complete Theorem \ref{thm:all_k}, we are left with the cases $k=1,2,3$. We start with $k=1$.

\begin{proposition}\label{prop:k=1} 
The $5$-manifold $S^2\times S^3$ admits a negative Sasakian structure.
\end{proposition}

\begin{proof} 
Let $X=\CP^2\#{\overline{\CP}^2}$ be the blow-up of $\CP^2$ at a point $p$. 
The cohomology $H^2(X,\ZZ)=\ZZ\langle L, E\rangle$ is generated by a line $L\subset \CP^2$ and an exceptional divisor $E$. A standard calculation yields $K_X=-3L+E$. The ample cone is generated by $L-E,L$, and the effective cone is generated by $L-E,E$. 
Take a collection of general conics in $\CP^2$ through $p$, and denote $C_1, \ldots ,C_s \subset X$ the proper transforms. 
In homology, $[C_i]=2L-E$ and $C_i,C_j$ intersect transversally at $3$ points (outside the exceptional divisor).
Introduce a structure of a smooth orbifold on $X$ by choosing a branch divisor
  $$
  \Delta=\sum_{i=1}^s \left(1-\frac{1}{m_i}\right) C_i,
  $$
where $m_i$ are pairwise coprime integer numbers.
By the general formula for the orbifold canonical class we get
  \begin{equation} \label{eqn:can-class} 
  K_X^{\orb}=-3L+E+\sum_{i=1}^s \left(1-\frac{1}{m_i}\right)(2L-E).
 \end{equation}

As in all previous considerations we use  Proposition \ref{prop:c1} and Theorem \ref{thm:H1(M)}, 
and construct a Seifert bundle determined by the orbifold structure on $X$ and an orbifold Chern class $c_1(M)=[\omega]$.  This will prove the Proposition provided that $K_X^{\orb}>0$. This condition is checked by a direct calculation of the inner product of $K_X$ with elements of the effective cone:
  $$
  K_X^{\orb}\cdot (L-E)>0, \;\;\; K_X^{\orb}\cdot E>0.
  $$
%provided that we use $s$ large enough.
Substituting (\ref{eqn:can-class}) to these inequalities, we see that they are satisfied for sufficiently large $s$ and $m_i$. 
For example, the first inequality yields $-2+\mathop{\sum}\limits_{i=1}^s (1-\frac{1}{m_i})>0$, so $s\geq 3$ and large $m_i$ will do. 
The second equality says $-1+\mathop{\sum}\limits_{i=1}^s (1-\frac{1}{m_i})>0$.

Consider the Seifert bundle $M\rightarrow X$ determined by the local invariants $(m_i,b_i)$ and the orbifold first Chern class  
  \begin{equation*} %\label{eqn:chern=1} 
  c_1(M)=c_1(B)+\sum_{i=1}^s\frac{b_i}{m_i}[C_i].
  \end{equation*}
We need to check assumptions (2) and (3) of Theorem \ref{thm:H1(M)}. Since $C_i\cdot E=1$, the canonical map $H^2(X,\ZZ)
\rightarrow H^2(C_i,\ZZ)$ sends $[E]\rightarrow E\cdot C_i=1$, therefore the induced map
$H^2(X,\ZZ)\rightarrow H^2(C_i,\ZZ_{m_i})$ is onto. As $m_i$ are coprime, we see that (2) is satisfied. 
For checking assumption (3), 
note that the K\"ahler forms are $[\omega]=a_1(L-E)+a_2L$, 
$a_1,a_2>0$. Take a line bundle $B$ and write $c_1(B)=\beta_1(L-E)+\beta_2L$, $\beta_1,\beta_2\in \ZZ$, and calculate
  \begin{equation}\label{eqn:chern=1} 
   c_1(M)= 
 \left(\beta_1+ \sum \frac{b_i}{m_i}\right)[L-E ] + \left(\beta_2+ \sum \frac{b_i}{m_i}\right)[L ].
  \end{equation}
This is a K\"ahler form if we arrange both coefficients positive. For the first one we take $\beta_1=0$.
Now $\mu=m_1\cdots m_s$ and so the second coefficient of $c_1(M/\mu)=\mu c_1(M)$ is adjusted to be
 \begin{equation}\label{eqn:prim} 
 \beta_2 m_1\cdots m_s + \sum_{i=1}^s b_i m_1\cdots\hat m_i\cdots m_s =1.
 \end{equation}
This can be solved since the $m_i$ are pairwise coprime. Note that we can take $0<b_i<m_i$, $\gcd(b_i,m_i)=1$.
Hence $c_1(M)$ can be represented by a K\"ahler form and $c_1(M/\mu)$ is primitive.

Finally, we need to show that $\pi_1^{\orb}(X)$ is trivial, which is proved applying Proposition \ref{prop:nori}, because $C_i^2>0$. 
We complete the proof noting that $M=S^2\times S^3$ by Theorem \ref{thm:H1(M)}, the classification of Smale-Barden manifolds and Theorem \ref{thm:w2}.
 \end{proof}

\begin{proposition}\label{prop:k=2,3} 
The Smale-Barden manifolds $\#_2(S^2\times S^3)$ and $\#_3(S^2\times S^3)$ admit negative Sasakian structures.
\end{proposition}

\begin{proof} 
The steps of the proof are similar to those in Proposition \ref{prop:k=1}, 
so we only explain the necessary modifications. For the case $\#_3(S^2\times S^3)$, we 
take $X=\CP^2\#3\overline{\CP}^2$, by blowing-up $\CP^2$ at three different non-collinear points. 
We write $H^2(X,\ZZ)=\ZZ\langle L, E,E',E''\rangle$, where $E,E',E''$ denote the exceptional divisors.
Now the effective cone is generated by
   $$
   E,E',E'' \,\,\text{and}\,\,L-E-E', L-E-E'', L-E'-E''
   $$
 (these are classes of the proper transforms of the lines through each two of the chosen blow-up points). Thus, the conditions ensuring ampleness of a cohomology class $aL-bE-cE'-dE''$ can be calculated by taking the inner product with generators of the effective cone. 
 %(in a way similar to Proposition \ref{prop:k=1}). 
 We get the ampleness conditions as the inequalities:
\begin{equation}\label{eqn:ampleness}
b>0,\;\; c>0,\;\; d>0, \;\; a>b+c,\;\; a>c+d, \;\; a>d+b.
\end{equation}

Take the strict transforms  $C_1,C_2,C_3$ of $3$ conics forming a configuration such that each conic passes through two of the three chosen points, but no two conics pass through exactly the same points
(and the other 3 intersections are transverse and generic). This yields
 \begin{equation} \label{eqn:ample3} 
 [C_1]=2L-E-E', \;\;\; [C_2]=2L-E' - E'', \;\;\; [C_3]=2L-E-E''.
 \end{equation}
Endow $X$ with a structure of a smooth orbifold with branch divisor
 $$
  \Delta=\left(1-\frac{1}{m_1}\right)C_1+\left(1-\frac{1}{m_2}\right)C_2+\left(1-\frac{1}{m_3}\right)C_3.
  $$
As $K_X^{\orb} =-3L+E+E'+E''+ \Delta$, a direct calculation using (\ref{eqn:ample3}) yields
 $$
 K_X^{\orb}=\Big(3-2\sum_{i=1}^3\frac{1}{m_i}\Big)L-
 \Big(1-\frac{1}{m_2}-\frac{1}{m_3}\Big)E-\Big(1-\frac{1}{m_1}-\frac{1}{m_3}\Big)E'-\Big(1-\frac{1}{m_1}-\frac{1}{m_2}\Big)E''.
 $$
Comparing this with (\ref{eqn:ampleness}) we see that the ampleness condition is satisfied for sufficiently large $m_i$ so that
 $$
 \frac{1}{m_1}+\frac{1}{m_2}<1,\;\;\;  \frac{1}{m_1}+\frac{1}{m_3}<1,\;\;\; \frac{1}{m_2}+\frac{1}{m_3}<1.
 $$
The coprimality of $m_i$ ensures that $H^2(X,\ZZ)\rightarrow H^2(C_i,\ZZ_{m_i})$ is surjective,  in the same way as in the proof of Proposition \ref{prop:k=1}.
 
Next we need to arrange that $\mu c_1(M)$ is primitive and ample (represented by a 
K\"ahler  form). For this take $c_1(B)=\beta_1 [L]$, and solve so that
the coefficient of $L$ in $\mu\, c_1(M)$ is $1$ in a similar vein to (\ref{eqn:prim}). 
Finally, the orbifold fundamental group is calculated as before, using $ C_i^2>0$. As before, we know that $M=\#_3(S^2\times S^3)$.

The case $\#_2(S^2\times S^3)$ is similar but easier, and it is left to the reader.
\end{proof}

Finally, we just recall that Propositions \ref{prop:large_k}, \ref{prop:k=1} and \ref{prop:k=2,3} complete the proof of Theorem \ref{thm:all_k}.

The constructions of
Propositions \ref{prop:k=1} and \ref{prop:k=2,3} yield semi-regular Sasakian structures. It is likely that an 
extension of this construction may cover also the cases $k\geq 4$, although it can be more involved.
On the other hand, the construction in 
Proposition \ref{prop:large_k} yields quasi-regular Sasakian structures with no branch divisor.

It is natural to ask the following:
\begin{question}
 Do the manifolds $\#_k (S^2\times S^3)$ admit {\it regular} negative Sasakian structures?
 \end{question}

As mentioned in the introduction, the values $k=(d-2)(d^2-2d+2)+1$, $d\geq 5$, $k=7,12,20$ are known.
We have some more cases.

\begin{theorem}\label{thm:low-betti} 
The $5$-manifold $\#_k(S^2\times S^3)$ admits a regular negative Sasakian structures
for $k=5,6,7, 8$ and $k=13$, $15\leq k\leq 20$.
\end{theorem}

\begin{proof}
By \cite[Lemma 9.10]{Ham} and \cite[Theorem 9.12]{Ham}, if $X$ is a simply connected $4$-manifold and 
$M\rightarrow X$ be a circle bundle over $X$ with primitive Euler class $e$ and
$w_2(X)\equiv 0,e \pmod2$,  then $M=\#_k(S^2\times S^3)$, where $k=b_2(X)-1$.
Therefore it is sufficient  to construct a simply connected 
smooth complex surface of general type
 and with Betti numbers $b_2(X)=k+1$ (compare the discussion in \cite[Example 10.4.6]{BG}). 
 This is because $K_X>0$, and Theorem \ref{thm:w2} applies. In particular, there is no need to check the condition $w_2(X)\equiv 0,e \pmod 2$. The case $k=8$ is already known: $X$ is the Barlow surface (see \cite[Example 10.4.6]{BG}).
The surfaces with $b_2(X)=6,7,8$ can be found in \cite{LP} and \cite{PPS}. This can be seen directly from the construction of these surfaces, 
because they are constructed using rational blow-down theory and $\mathbb{Q}$-Gorenstein smoothing theory. 
In more detail: in \cite{LP} and \cite{PPS} the authors construct simply connected smooth complex surfaces of general type with $p_g=0$ and $K_X^2=1,2,3,4$. Since $X$ is simply connected, the Noether formula  takes the form
 $$
 \frac{K_X^2+\chi(X)}{12}=1,
 $$
and one obtains $b_2(X)=\chi(X)-2=6,7,8,9$.
 
The cases $k=13$ and $15\leq k\leq 20$ follow from \cite{PPS1} in the same way. In greater detail, \cite{PPS1} yields examples of smooth
simply connected surfaces of general type $X$ with invariants $K_X^2=1,2,\ldots ,6,8$ and $p_g=1,q=0$. 
Again, the Noether formula yields the desired values of $b_2(X)=k+1$ and $b_2(M)=k$, 
where $M$ is the total space of the circle bundle whose first Chern class is the cohomology class of the K\"ahler form $[\omega]$. 
\end{proof}

\end{document}